\documentclass[12pt,leqno]{amsart}

\usepackage{amsfonts}
\usepackage{amssymb}
\usepackage{amsmath}
\usepackage{amscd}
\usepackage{amstext}
\usepackage{ifthen}
\usepackage[all]{xy}
\usepackage{enumerate}
\usepackage{fancyhdr}   
\usepackage[pdftex]{color}

\makeatletter
\def\@cite#1#2{{\m@th\upshape\bfseries%
[{#1\if@tempswa{\m@th\upshape\mdseries, #2}\fi}]}}
\makeatother

\newtheorem{theorem*}{Theorem}
\newtheorem{theorem}{Theorem}[section]

\newtheorem{corollary}[theorem]{Corollary}
\newtheorem{proposition}[theorem]{Proposition}

\theoremstyle{definition}

\newtheorem{remark}[theorem]{Remark}

\numberwithin{equation}{section}

  \newcommand{\A}{{\mathcal{A}}}
  \newcommand{\B}{{\mathcal{B}}}

  \newcommand{\E}{{\mathcal{E}}}
  \newcommand{\F}{{\mathcal{F}}}

  \newcommand{\N}{{\mathcal{N}}}

\renewcommand{\phi}{\varphi}
\newcommand{\upchi}{{\raise.35ex\hbox{\ensuremath{\chi}}}}

\newcommand{\ad}{\operatorname{Ad}}

\newcommand{\id}{{\operatorname{id}}}

\newcommand{\m}{\operatorname{m}}

\newcommand{\li}{\operatorname{-lim}}

\begin{document}

\title[]{Near inclusions of amenable operator algebras}

\author[J. Roydor]{Jean Roydor}

\thanks{The author is supported by JSPS}

\maketitle

\begin{abstract}
We prove that if an amenable operator algebra is nearly contained in a complemented dual operator algebra, then it can be embedded inside this dual operator algebra via a similarity. The proof relies on a B.E. Johnson Theorem on approximately multiplicative maps.
\end{abstract}

\maketitle

\section{Introduction}

In \cite{KK}, R.V. Kadison and D. Kastler defined a metric on the collection of all subspaces of the bounded operators on a fixed Hilbert space. They conjectured that sufficiently close von Neumann algebras (or $C^*$-algebras) are necessarily unitarily conjugate. A great amount of work around this perturbation problem has been achieved in the last forty years (see e.g. \cite{C0}, \cite{C1}). Very recently, Kadison-Kastler's conjecture has been proved for the class of separable nuclear $C^*$-algebras in the remarkable paper \cite{C5} (see also \cite{C4}). One-sided versions of Kadison-Kastler's conjecture have been considered as well ; in \cite{C2}, E. Christensen proved that a nuclear $C^*$-algebra which is nearly contained in an injective von Neumann algebra can be unitarily conjugated inside this von Neumann algebra. In this short note, we prove an analog of this near inclusion result for non-selfadjoint operator algebras.\\
 \indent The emergence of operator space theory, in the late eighties, gave a nice framework to study non-selfadjoint operator algebras (see \cite{BLM}, \cite{ER1}, \cite{Pa} and \cite{P}). An important conjecture in this area has been raised by G. Pisier: a non-selfadjoint operator algebra which is amenable (as a Banach algebra) should be similar to a nuclear $C^*$-algebra (see \cite{BLM} Section 7.4 for more details). With this conjecture in mind, one can expect that a version of E. Christensen's aforementioned result must be true with an amenable operator algebra. More precisely, an amenable operator algebra which is nearly contained in an injective von Neumann algebra should embed inside this von Neumann algebra by a similarity. Actually, we prove slightly more (see Theorem \ref{T:jr1.5} below), because we do not assume the `container algebra' to be selfadjoint and we just need a bounded projection on it (not necessarily contractive).\\
  \indent The difficulty is that non-selfadjoint algebras are much more flexible than selfadjoint ones and tools to study them are rather limited (no order structure, no continuous functional calculus, and no abundance of projections or Borel functional calculus in the dual case). For instance, one very important step in the resolution of perturbation problems (referred as the second step in the introduction of \cite{C2}) is to perturb a linear isomorphism close to the identity map into a $*$-homomorphism (see Lemma 3.3 in \cite{C1} or Lemma 3.2 in \cite{C5}). E. Christensen's averaging trick is not available in our present case ; instead we will use a Theorem of B.E. Johnson on approximately multiplicative maps, this is the main new ingredient. We will also need to prove a non-selfadjoint analog of Theorem 5.4 in \cite{C1} on neighboring representations of amenable operator algebras (see the third step mentioned in the introduction of \cite{C2}).\\

\section{Preliminaries}\label{Pre}

\subsection{Virtual diagonal} In \cite{J0}, B.E. Johnson defined the notion of amenability for Banach algebras in cohomological terms. Subsequently, he proved that a Banach algebra is amenable if and only if it admits a virtual diagonal. Let us recall this notion of virtual diagonal.\\
Let $\A$ be a unital Banach algebra. The projective tensor product $\A \widehat{\otimes} \A$ can be equipped with a Banach $\A$-bimodule structure, for any $a,b,x,y \in \A$:
$$a \cdot (x \otimes y)\cdot b= ax \otimes yb.$$
Thus, the bidual space $(\A \widehat{\otimes} \A)^{**}$ can be turned canonically into a Banach $\A$-bimodule as well (by duality).
A \textit{virtual diagonal for $\A$} is an element $u \in (\A \widehat{\otimes} \A)^{**}$ satisfying:
\begin{enumerate}
\item for any $a \in \A$, $a \cdot u= u \cdot a$,
\item $\m_\A ^{**}(u)=1$ (where $\m_\A:\A \widehat{\otimes} \A \to \A$ denotes the multiplication map).
\end{enumerate}

\subsection{Approximately multiplicative maps} Let us recall a Theorem of B.E. Johnson on approximately multiplicative maps, which will be a crucial ingredient in the proof of the main result. In \cite{J}, B.E. Johnson proved that an approximately multiplicative map defined on an amenable Banach algebra is close to an actual algebra homomorphism. His result is the Banach algebraic version of an earlier result due to D. Kazhdan for amenable groups (see \cite{K}).
If $L$ is a linear map between Banach algebras $\A$ and $\B$, we denote $L^{\vee}:\A \times \A \to \B$ the bilinear map defined by \begin{equation}\label{aeq}L^{\vee}(x,y)=L(xy)-L(x)L(y).\end{equation} This enables us to measure the defect of multiplicativity of $L$.\\
 For the following Theorem, we recall that a Banach algebra $\B$ is called a \textit{dual Banach algebra} if there is a sub-$\B$-module $\B_*$ of $\B^*$ such that $\B=(\B_*)^*$. Note that an operator algebra which is a dual Banach algebra in this sense is actually a dual operator algebra in the sense of Section 2.7 in \cite{BLM} (i.e. a $w^*$-closed subalgebra of a certain $\mathbb{B}(H)$, the von Neumann algebra of all bounded operators on a Hilbert space $H$).

\begin{theorem}[(\cite{J}, Th. 3.1)]\label{T:j} Let $\A$ be a unital amenable Banach algebra and suppose that $\B$ is a dual Banach algebra. Then, for any $\varepsilon \in ]0,1[$, for any $\mu >0$, there exists $\delta >0$ such that:\\
 for every unital bounded linear map $L:\A \to \B$ satisfying $\Vert L \Vert \leq \mu$ and $\Vert L^{\vee} \Vert \leq \delta$, there is a unital bounded algebra homomorphism $\pi:\A \to \B$ such that $\Vert L - \pi \Vert \leq \varepsilon$.
\end{theorem}

The important point is that $\delta$ is explicit, actually one can choose
\begin{equation}\label{b3eq}\delta=\frac{\varepsilon}{4 \Vert u \Vert + 8\mu^2\Vert u \Vert^2},\end{equation}
where $u$ is a virtual diagonal for $\A$ (see the proof of Theorem 3.1 in \cite{J}).

\section{Amenability and neighboring representations}

The question whether neighboring representations are necessarily equivalent was already posed in \cite{KK}. Here we prove that two representations of an amenable operator algebra which are close enough are necessarily similar. When one of the representation is the identity representation, the Proposition below can be considered as parameterized version of Theorem 5.4 in \cite{C1}. Indeed, according to the work of U. Haagerup \cite{H} Theorem 3.1, we know that a nuclear $C^*$-algebra is amenable (as a Banach algebra) and more precisely, admits a virtual diagonal of norm one. Therefore in the $C^*$-case, the quantity $\Vert u \Vert ^{-1}  \max \{ \Vert \pi_1  \Vert ^{-1}, \Vert \pi_2  \Vert ^{-1} \}$ considered below, equals one.\\
The proof of the following Proposition is inspired from the averaging technique appearing in the proof of Lemma 3.4 of \cite{C5} (equality (\ref{eq}) in the following proof can be compared with equation (3.21) in \cite{C5}).\\
If $S$ is an invertible operator, we denote $\ad_S$ the similarity implemented by $S$.
\begin{proposition}\label{P:rep} Let $\A$ be a unital amenable operator algebra with virtual diagonal $u \in (\A \widehat{\otimes} \A)^{**}$. Let $\pi_1, \pi_2$ be two unital bounded representations on the same Hilbert space $K$.\\
If $$\Vert \pi_1 - \pi_2 \Vert<\Vert u \Vert ^{-1}  \max \{ \Vert \pi_1  \Vert ^{-1}, \Vert \pi_2  \Vert ^{-1} \},$$ then there exists an invertible operator $S$ in the $w^*$-closed algebra generated by $\pi_1(\A) \cup \pi_2(\A)$ such that $\pi_1=\ad_S \circ \pi_2$. Moreover, $$\Vert S-I_K \Vert \leq  \Vert u \Vert \Vert \pi_1 - \pi_2 \Vert \min \{ \Vert \pi_1  \Vert , \Vert \pi_2  \Vert  \},$$ where $I_K$ denotes the identity of $\mathbb{B}(K)$.
\end{proposition}
\begin{proof} We will use the following notation: for bounded linear maps $F,G:\A \to \mathbb{B}(K)$, we denote $\psi_{F,G}:\A \widehat{\otimes} \A \to \mathbb{B}(K)$  the linear map uniquely defined by $$\psi_{F,G}(x \otimes y)=F(x)G(y).$$ Hence $\psi_{F,G}$ is bounded and $\Vert \psi_{F,G} \Vert \leq \Vert F \Vert \Vert G \Vert$. Further, define $\Psi_{F,G}:(\A \widehat{\otimes} \A)^{**} \to \mathbb{B}(K)$ to be the unique $w^*$-continuous bounded extension of $\psi_{F,G}$. That is, $$\Psi_{F,G}=i^* \circ \psi_{F,G}^{**},$$ where $i:S^1(K) \hookrightarrow \mathbb{B}(K)^*$ denotes the canonical injection from the predual of $\mathbb{B}(K)$ inside its dual. Thus, $\Vert \Psi_{F,G} \Vert \leq \Vert F \Vert \Vert G \Vert$ as well.\\
 Let $\pi_1,\pi_2$ be as above. Define $$S=\Psi_{\pi_1,\pi_2}(u) \in \mathbb{B}(K).$$
  As $u \in (\A \widehat{\otimes} \A)^{**}$, there is a net $(u_t)_t$ in $\A \otimes \A$ converging to $u$ in the $w^*$-topology of $(\A \widehat{\otimes} \A)^{**}$. For any $t$, there are finite families $(a_k^t)_k,(b_k^t)_k$ of elements in $\A$ such that $$u_t=\sum_k a_k^t \otimes b_k^t.$$ Since $\Psi_{\pi_1,\pi_2}$ is $w^*$-continuous,
  \begin{equation}\label{eq}S=w^*\li _t \Psi_{\pi_1,\pi_2}(u_t)
  =w^*\li _t \sum_k \pi_1(a_k^t)\pi_2(b_k^t),
  \end{equation}
 which shows that $S$ lies in the $w^*$-closed algebra generated by $\pi_1(\A) \cup \pi_2(\A)$.\\
 Now let's prove that $S$ intertwines $\pi_1$ and $\pi_2$. Fix $a \in \A$, then
 \begin{align*}\pi_{1}(a)S=& \pi_{1}(a).w^*\li _t \sum_k \pi_1(a_k^t)\pi_2(b_k^t)\\
=&w^*\li _t \sum_k \pi_1(aa_k^t)\pi_2(b_k^t) \\
 = & w^*\li _t \Psi_{\pi_1,\pi_2}(a\cdot u_t)\\
 = & \Psi_{\pi_1,\pi_2}(a\cdot u).
 \end{align*}
 Analogously, we can show that $S \pi_2(a)=\Psi_{\pi_1,\pi_2}(u \cdot a)$. But $u$ is a virtual diagonal for $\A$, so $a\cdot u=u \cdot a$, hence $$\pi_{1}(a)S=S \pi_2(a).$$
To finish the proof, we just need to prove that $S$ is invertible. Without loss of generality we can assume that $\max \{ \Vert \pi_1  \Vert ^{-1}, \Vert \pi_2  \Vert ^{-1} \}= \Vert \pi_1  \Vert ^{-1}$, so  $\Vert \pi_1 - \pi_2 \Vert<\Vert u \Vert ^{-1} \Vert \pi_1  \Vert ^{-1}$.
 As above, one can check that $$\Psi_{\pi_1,\pi_1}(u)= w^*\li _t \sum_k \pi_1(a_k^t)\pi_1(b_k^t).$$ Then using the condition (2) defining a virtual diagonal, we obtain
\begin{align*}\Psi_{\pi_1,\pi_1}(u)=& w^*\li _t \sum_k \pi_1(a_k^tb_k^t)\\
=&\pi_1^{**}(w^*\li _t \sum_k a_k^tb_k^t) \\
 = & \pi_1^{**}(\m_\A ^{**}(u))\\
 =& \pi_1(1)\\
 =&I_K.\end{align*}
Consequently,
\begin{align*}\Vert S - I_K \Vert=& \Vert \Psi_{\pi_1,\pi_2}(u)- \Psi_{\pi_1,\pi_1}(u) \Vert\\
=&\Vert \Psi_{\pi_1,\pi_2-\pi_1}(u) \Vert \\
 \leq & \Vert u \Vert \Vert \pi_1  \Vert \Vert \pi_2 - \pi_1 \Vert \\
  < & 1,\end{align*}
 which shows that $S$ is invertible.
\end{proof}
\begin{remark}
This proposition can be compared with Theorem 5.1 in \cite{J1} or Theorem 2 in \cite{RT}. Here, our advantage is that the bound controlling the distance between the two homomorphisms is explicit, whereas the proofs of Theorem 5.1 in \cite{J1} or Theorem 2 in \cite{RT} use the open mapping Theorem to obtain this bound.
\end{remark}

\section{Proof of the main Theorem}

We recall the notion of near inclusion. Let $\gamma>0$, $\E,\F \subset \mathbb{B}(H)$ be two subspaces. We write $\E \subseteq^\gamma \F$  if for any $x$ in the unit ball of $\E$, there exists $y$ in $\F$ such that $$\Vert x-y \Vert \leq \gamma.$$
In this case, $\E$ is said to be \textit{nearly contained in $\F$ with constant} $\gamma$.\\
In the following Theorem, we denote by $\digamma$ the polynomial function
$$\digamma(x,y)=(1+ x)(y+ 4(x+2)y^2+ 8(x+2)x^2 y^{3}).$$

\begin{theorem}\label{T:jr1.5} Let $\A,\N \subset \mathbb{B}(H)$ be two unital operator subalgebras. Suppose that $\A$ is amenable (with a virtual diagonal $u$). Assume that $\N$ is $w^*$-closed and there is a bounded projection $P$ from $\mathbb{B}(H)$ onto $\N$.\\
If $\A \subseteq^{\gamma} \N$, with \begin{equation}\label{a0eq}\gamma < \frac{1}{\digamma(\Vert P \Vert,\Vert u \Vert)},\end{equation} then there exists an invertible operator $S$ in the $w^*$-closed algebra generated by $\A \cup \N$ such that $S\A S^{-1} \subset \N$. Moreover, $$\Vert S-I_H\Vert \leq  \digamma(\Vert P \Vert,\Vert u \Vert) \gamma.$$
\end{theorem}

\begin{proof}
Denote $T=P_{\vert \A}:\A \to \N$. Let $x$ be in the unit ball of $\A$, then there is $z$ in $\N$ such that $\Vert x - z \Vert \leq \gamma$. Then
\begin{align*}\Vert T(x) - x \Vert &=\Vert T(x-z) - (x-z) \Vert \\
&\leq (1+\Vert P \Vert)\gamma,\end{align*}
which means that \begin{equation}\label{a1eq}\Vert T-\id_\A \Vert \leq (1+\Vert P \Vert)\gamma\end{equation}
(where $\id_\A:\A \hookrightarrow \mathbb{B}(H)$ denotes the identity representation). Let's compute the defect of multiplicativity of $T$ (see equation (\ref{aeq})). Let $x,y$ be in the unit ball of $\A$

\begin{align*}
\Vert T^\vee(x,y) \Vert&=\Vert T(xy)-T(x)T(y) \Vert \\
&\leq \Vert T(xy)-xy \Vert+\Vert xy-xT(y) \Vert+ \Vert xT(y)- T(x)T(y) \Vert\\
&\leq \Vert T-\id_\A \Vert +\Vert T-\id_\A \Vert + \Vert T-\id_\A \Vert \Vert T \Vert\\
&\leq (2+\Vert P \Vert)\Vert T-\id_\A \Vert.
\end{align*}
Hence from equation (\ref{a1eq}), we obtain \begin{equation}\label{a2eq}\Vert T^\vee \Vert \leq (2+\Vert P \Vert)(1+\Vert P \Vert)\gamma.\end{equation}
We want to apply Theorem \ref{T:j} to the linear map $T$ (valued in the dual operator algebra $\N$). Let $\mu=\Vert P \Vert$, $$\delta=(2+\Vert P \Vert)(1+\Vert P \Vert)\gamma$$ and (see equation (\ref{b3eq}))
\begin{align*}
\varepsilon&=\delta(4 \Vert u \Vert + 8\mu^2\Vert u \Vert^2) \\
&=(1+ \Vert P \Vert)(4(\Vert P \Vert+2)\Vert u \Vert+ 8(\Vert P \Vert+2)\Vert P \Vert^2 \Vert u \Vert^{2})\gamma.
\end{align*}
Note first that
\begin{equation}\label{eq1} \varepsilon + (1+\Vert P \Vert)\gamma = \frac{\digamma(\Vert P \Vert,\Vert u \Vert) \gamma}{\Vert u \Vert},\end{equation}
so from the condition $(\ref{a0eq})$ on $\gamma$, we have \begin{equation}\label{a3eq}\varepsilon + (1+\Vert P \Vert)\gamma < \Vert u \Vert^{-1}.\end{equation}
Further, as the norm of a virtual diagonal is always greater or equal to one (from condition (2) defining a virtual diagonal, see Preliminaries Section), $\varepsilon$ is strictly smaller than one.
Consequently, from equations (\ref{b3eq}) and (\ref{a2eq}), we can apply Theorem \ref{T:j} to the linear map $T$. Thus there exists a unital bounded algebra homomorphism $\pi:\A \to \N$ such that $\Vert T - \pi \Vert \leq \varepsilon$.\\
Now we want to apply Proposition \ref{P:rep} to $\pi$ and $\id_\A$. Clearly, $\max \{ \Vert \pi  \Vert ^{-1}, \Vert \id_\A  \Vert ^{-1} \}=1.$
Moreover, from equation (\ref{a1eq}), we have
\begin{align*}
\Vert \pi-\id_\A \Vert & \leq \Vert \pi-T \Vert +\Vert T-\id_\A \Vert\\
& \leq \varepsilon + (1+\Vert P \Vert)\gamma
\end{align*}
  and this quantity is strictly smaller than $\Vert u \Vert^{-1}$ from equation (\ref{a3eq}) above. Applying Proposition \ref{P:rep}, we obtain a similarity $S$ such that $\ad_S \circ \id_\A= \pi$ and
 \begin{align*}
\Vert S - I_H \Vert & \leq \Vert u \Vert \Vert \pi-\id_\A \Vert \\
& \leq \Vert u \Vert (\varepsilon + (1+\Vert P \Vert)\gamma)
\end{align*}
and this quantity is smaller than $\digamma(\Vert P \Vert,\Vert u \Vert) \gamma$ by equation (\ref{eq1}).
\end{proof}

\begin{remark}
Since the norm of a virtual diagonal and the norm of a projection are always greater or equal to one, necessarily $\gamma < 1/74$, in the preceding Theorem.
\end{remark}

From \cite{H}, we know that a nuclear $C^*$-algebra admits a virtual diagonal of norm one. Combining this result with Lemma 2.7 in \cite{C0}, we obtain the following Corollary (here, $\digamma_P=\digamma(\Vert P \Vert,1)$):

\begin{corollary} Let $\A,\N \subset \mathbb{B}(H)$ be two unital subalgebras. Suppose that $\A$ is a nuclear $C^*$-algebra. Assume that $\N$ is a von Neumann subalgebra which is the range of bounded projection $P$.\\
If $\A \subseteq^{\gamma} \N$, with $$\gamma < \frac{1}{\digamma_P},$$ then there exists a unitary $U$ in the von Neumann subalgebra generated by $\A \cup \N$ such that $U\A U^{*} \subset \N$. Moreover, $$\Vert U-I_H\Vert \leq  \frac{\sqrt{2} \digamma_P }{\sqrt{1+\sqrt{1-\digamma_P^2\gamma^2}}}\gamma.$$
\end{corollary}

\begin{remark}Since we just assume $P$ bounded (not necessarily contractive), this previous Corollary might improve slightly Corollary 4.2 (a) in \cite{C2}.
However, it is an open problem whether a von Neumann algebra which is the range of a bounded projection is necessarily injective (see \cite{Pis}).
\end{remark}

\textbf{Acknowledgements.} The author would like to thank Stuart White for introducing him to perturbation theory of operator algebras. The author is grateful to Narutaka Ozawa for hosting him at the University of Tokyo.


\email{Jean Roydor, Department of Mathematical Sciences, Tokyo, 153-8914, Japan.\\
roydor@ms.u-tokyo.ac.jp}

\end{document}